\documentclass[a4paper, 12pt,reqno]{amsart}

\usepackage[T1]{fontenc}
\usepackage[utf8]{inputenc}
\usepackage[english]{babel}
\usepackage{url} 
\usepackage{amsmath}
\usepackage{amssymb}
\usepackage{amsthm}
\usepackage{mathtools}
\usepackage[italian]{varioref}
\usepackage{hyperref}
\usepackage{typehtml}
\usepackage{multicol}
\usepackage{color}
\usepackage{mathrsfs}
\usepackage{comment}
\usepackage{mathabx}
\usepackage{tikz-cd}
\usepackage{adjustbox}
\usepackage{bbm}
\usepackage{hyperref}
\usepackage{scalerel,stackengine}
\stackMath
\newcommand\reallywidehat[1]{%
\savestack{\tmpbox}{\stretchto{%
  \scaleto{%
    \scalerel*[\widthof{\ensuremath{#1}}]{\kern-.6pt\bigwedge\kern-.6pt}%
    {\rule[-\textheight/2]{1ex}{\textheight}}
  }{\textheight}%
}{0.5ex}}%
\stackon[1pt]{#1}{\tmpbox}%
}
\hypersetup{hidelinks}

\newcommand{\R}{\mathbb R}

\newcommand{\Z}{\mathbb Z}
\newcommand{\N}{\mathbb N}

\newcommand{\imm}{\mathrm{i}}           
\newcommand{\e}{\mathrm{e}}                
\newcommand{\de}{\mathop{}\!\mathrm{d}}

\newcommand{\eps}{\varepsilon}
\newcommand{\vertiii}[1]{{\left\vert\kern-0.25ex\left\vert\kern-0.25ex\left\vert #1 
    \right\vert\kern-0.25ex\right\vert\kern-0.25ex\right\vert}}

\theoremstyle{definition}

\theoremstyle{plain}
\newtheorem{Theorem}{Theorem}[section]
\newtheorem{lemma}{Lemma}[section]

\newtheorem{Proposition}{Proposition}[section]
\theoremstyle{remark}
\newtheorem{Remark}{Remark}[section]

\numberwithin{equation}{section}

\title[Forward vs backward
Landau damping]{Comparison between the Cauchy problem and the scattering problem for the Landau Damping in the Vlasov-HMF equation}

\author[D.\ Benedetto]{Dario Benedetto}
\address{Dario Benedetto \hfill\break \indent
  Dipartimento di Matematica, Universit\`a di Roma `La Sapienza'
  \hfill\break \indent
  P.le Aldo Moro 2, 00185 Roma, Italy}
\email{benedetto@mat.uniroma1.it}

\author[E.\ Caglioti]{Emanuele Caglioti}
\address{Emanuele Caglioti \hfill\break \indent
  Dipartimento di Matematica, Universit\`a di Roma `La Sapienza'
  \hfill\break \indent
  P.le Aldo Moro 2, 00185 Roma, Italy}
\email{caglioti@mat.uniroma1.it}

\author[S.\ Rossi]{Stefano Rossi}
\address{Stefano Rossi\hfill\break \indent
  Dipartimento di Matematica, Universit\`a di Roma `La Sapienza'
  \hfill\break \indent
  P.le Aldo Moro 2, 00185 Roma, Italy}
\email{stef.rossi@uniroma1.it}

\begin{document}
\date{\today}
\begin{abstract}
  We analyze the analytic Landau damping problem 
  for the Vlasov-HMF equation, by fixing the asymptotic
  behavior of the solution. We use a new method
  for this ``scattering problem'', closer to the one
  used for the Cauchy problem. In this way we are able to compare
  the two results, emphasizing the different influence 
  of the plasma echoes in the two approaches.
  In particular, we 
  prove a non-perturbative
  result for the scattering problem.
\end{abstract}


\maketitle

\pagestyle{headings}

\section{Introduction}
In the spatially periodic case,
the Vlasov-Poisson equation in the HMF approximation
reads as
\begin{equation}
  \label{HMF}
  \partial_t f(t,x,v) + v \partial_{x} f(t,x,v)+\mathcal{F}[f](t,x) \partial_{v}f(t,x,v)=0,
\end{equation}
where
\begin{equation}
  \label{eq:campo}
  \mathcal{F}[f](t,x) =-
  \partial_x \Biggl{(}\int_{S^1 \times \mathbb{R}}\cos(x-y)f(t,y,v)
  \, dy \, dv \Biggr{)} 
\end{equation}
is the mean-field force. Here $f(t,x,v)$ is the normalized
density of electrons
with position $x\in S^1$ and velocity $v \in \R$,
in a collisionless electrically neutral plasma.

This model has been widely studied in the last decades being a handy
reduction of the Vlasov-Poisson equation, in which 
the singularity of the kernel is removed by replacing it with a cosine function.
It can be easily implemented numerically to study the features of a
long-range interaction (see \cite{AR}, \cite{YBBDR},
\cite{FR}). Furthermore,
this model is also a useful testing ground from a mathematical point of view
for studying issues about long-time behavior of solutions.
This is the case of the Landau damping, i.e. the existence of
damped solutions near a stationary regular state. 
The damping consists
in the existence of $\omega(x,v)$ such that
\begin{equation}
  \label{eq:limrh}
  \lim_{t\to+\infty} (f(t,x,v) - \omega (x-vt,v)) = 0,
\end{equation}
which means that the flow governed by the mean-field force
is asymptotically free, and $f(t,x,v)$ converges weakly to the
mean of $\omega$ in the $x$ variable.

After Landau's pioneering work of 1946 \cite{L} for the linearized Vlasov
equation, the damping phenomenon for mean-field models has
been extensively studied and understood in the last decades.
The first result for the nonlinear Vlasov-Poisson
problem is proved by Caglioti
and Maffei in \cite{CM}. They read the problem as a scattering
problem for the flow, 
by fixing the asymptotic datum $\omega$ and finding a solution
$f(t,x,v)$ satisfying \eqref{eq:limrh}.
Subsequently, a proof with less restrictive hypotheses was given in \cite{HV}.

In \cite{MV} Mouhot and Villani, introducing new mathematical techniques,
solve the Cauchy problem for the nonlinear Vlasov-Poisson equation,
with analytic and Gevrey initial data,
and show the existence of the asymptotic state $\omega$.
A substantial analogy exists between the Landau damping in plasma physics
and the inviscid damping for the two-dimensional Euler equation.
In fact in \cite{BM} the damping near the Couette flow
has been proved using different techniques, this gives rise to a
new simpler proof of the Landau damping result in \cite{BMM} 
(see also the recent result in \cite{GNR} for a more elementary proof).
For what concern the damping with Sobolev regularity, it has been
shown by Lin and Zeng (\cite{LZ}, \cite{LZ2}) that for very low
regularities
Landau damping cannot occur. Although, in the case of the Vlasov-HMF
equation with sufficiently high Sobolev regularity,
Faou and Rousset in \cite{FR} have succeeded in proving the damping
with a polynomial rate.  
A Landau damping result for the full Vlasov-Poisson equation with Sobolev data is still missing,
however Bedrossian in \cite{B} has given a negative answer to the possibility
of a straightforward extension to this setting of Mouhot and Villani's work in \cite{MV}.

The ``backward'' approach, which provides
the solution of the scattering problem
with a given $\omega$, 
and the ``forward'' approach, which provides the solution for the Cauchy problem
with $f_0(x,v)=f(0,x,v)$, are different
from many points of view, starting from the technical ones:
in the backward approach, as in \cite{CM} and \cite{HV}
(and also in \cite{BCM}),
using a Lagrangian point of view, it is proved that the flow
is close to the free one.
In this work, instead, in the HMF approximation,
we adapt the forward techniques to the backward problem
to make a comparison in the case of analytic solutions.
In particular, we discuss the different way the two approaches
overcome the difficulties due to the presence of the ``echoes'',
i.e. resonances at certain times between the Fourier modes of the solution
(for an in-depth analysis of echoes in Vlasov-Poisson equation 
with analytic or Gevrey initial data, see again \cite{GNR}).
This highlights a simplified structure of the norms used
in the backward approach.
Moreover, the backward technique
is unable to identify initial data for which damping occurs,
but works also in a non-perturbative regime, i.e.
without requiring the solution to be a small perturbation of a
stationary state.

In addition, as a by-product, we prove the backward nonlinear Landau damping
for the HMF equation, previously unknown.
Perhaps, this 
Eulerian approach can be applied also in the study
of the backward problem for other interesting models.

\vskip.3cm

The work is divided as follows: in Section 2 we prove the
Landau damping in a perturbative regime using the scattering approach.
We give {\it a priori} estimates in the time interval $[0, T]$
imposing that the solution reaches the asymptotic state at time $T$.
Then we send $T$ to infinity, obtaining the solution.
In Section 3, we reanalyze the problem in  $[\tau, T]$, with
$T\to+\infty$.
With a more subtle estimate of the echoes terms,
we obtain a non-perturbative existence
result for sufficiently large values of $\tau$.
In Section 4 we present a proof of the damping for the Cauchy problem
in order to highlight the differences with the backward approach,
which we deepen in Section 5.

In both approaches, we need to control the loss of analytic regularity of
the solutions. For this reason 
we use techniques inspired from the abstract Cauchy-Kovalevskaya theory
(see \cite{C}), adapted to this kind of problems in 
\cite{BCMD}.

\section{The scattering problem}

We consider solutions  of \eqref{HMF} which are small perturbations of
a spatially homogenous solution $\eta$, i.e.
\begin{equation}
  \label{eq:feer}
  f(t,x,v)=\eta(v) + \eps r(t,x,v),
\end{equation}
and we assume $\eta$ is an analytic function of the velocities.
The equation verified by the perturbation $r$ is
\begin{equation*}
\partial_t r(t,x,v) + v \partial_{x} r(t,x,v)+\mathcal{F}[r](t,x)
\partial_{v}\bigl{(}\eta(v)+ \eps r(t,x,v)\bigr{)}=0,
\end{equation*}
where the operator $\mathcal F$ is defined in \eqref{eq:campo}.

To state the asymptotic behavior as in \eqref{eq:limrh},
we define $h(t,x,v)=r(t,x+vt, v)$, which verifies the following equation:
\begin{equation}
\label{principale}
\partial_t h=\{ \psi[h],\eta+ \eps h \},
\end{equation}
where $\psi$ is the potential field generated
by the perturbation, evaluated along the free flow 
\begin{equation}
\label{campo}
\psi[h](t,x,v)=\int_{S_1 \times \mathbb{R}} \cos(x-y+(v-u)t)h(t,y,u) \, dy \, du
\end{equation}
and where $\{	 ,\}$ is the Poisson bracket.

Recalling \eqref{eq:limrh} and \eqref{eq:feer},
we study the
damping problem  by setting
$\omega(x,v) = \eta(v) + \eps h_\infty(x,v)$, i.e. 
by 
searching for a solution for \eqref{principale} such that
\[
\lim_{t \to +\infty}\|h(t,x,v) - h_{\infty}(x,v)\|_{\infty}=0
\]
where $h_{\infty}$ is a mean-zero analytic datum
with $\| h_{\infty} \|_{\lambda}<+\infty$ for some $\lambda>0$.

Firstly, we study the evolution in the time interval $[0,T]$
considering the following problem:
\begin{equation}
\begin{cases}
\label{tempoT}
\partial_t h^T(t,x,v)=\{ \psi[h^T],\eta+ \eps h^T \} \quad 0\le t\le T,\\
h^T(T,x,v)=  h_\infty(x,v).
\end{cases}
\end{equation}
Then, we show that, for $T\to +\infty$, $h^T$
converges to a solution $h$,  which solves the asymptotic problem.

We work in Fourier transform in $S_1\times \R$, using the following notation:
\[
  \widehat{g}_n(\xi )=\frac{1}{2\pi} \int_{S_1\times \R}
  \e^{-inx} \e^{-iv\xi} g(x,v) \, dx \, dv
\]
with $n \in \mathbb{Z}$ and $ \xi \in \mathbb{R}$.
In Fourier space the  system is
\begin{equation}
\label{differenziale}
\partial_t \widehat{h^T}_n(t, \xi)=\delta_{n, \pm 1}n\frac{\imm }{2}\zeta^T_n(t)\widetilde{\eta'}(\xi-nt)- \eps  \sum_{k= \pm 1}k\frac{\zeta^T_k(t)}{2}\widehat{h^T}_{n-k}(t, \xi-kt)(\xi -nt),
\end{equation}
where $\widetilde{\eta'}$ is the Fourier transform of $\eta'$ in the
velocity and  
$\zeta^T_n $ for $n = \pm 1$ is the electric field:
\begin{equation}
  \label{eq:zita}
  \zeta^T_n(t) = \widehat{h^T}_n(t, nt).
\end{equation}
Integrating equation \eqref{differenziale}
between $[t,T]$ and putting $\xi=nt$, we get an equation for $\zeta^T$:
\begin{multline}
  \label{zetaequa}
  \zeta^T_n(t)=
  \widehat{h^T}_n(T, nt)-\frac{\imm}{2}n \int_t^T \zeta^T_n(s)
  \widetilde{\eta'}(n(t -s))\, ds\\
  -
  \frac{\eps}{2} \sum_{k= \pm 1}\int_t^T \zeta^T_k(s)\widehat{h^T}_{n-k}
  (s, n t -ks)kn(s-t) \, ds.
\end{multline}
In order to give \emph{a priori} estimates, it is convenient to
consider $(\zeta^T_{\pm 1}, h^T)$ as a coupled
system, where \eqref{eq:zita} is a consequence of the uniqueness.

A key point in Landau damping problems
is the decay of the electric field.
To show this we define the norm of the electric field $\zeta^T$ as
\begin{equation}
\label{campoindietro}
M_{\lambda, T}[\zeta^T]=\sup_{t \in [0, T]} \e^{\lambda t}|\zeta_1^T(t)|=\sup_{t \in [0, T]} \e^{\lambda t}|\zeta_{-1}^T(t)|.
\end{equation}

We also define a norm which quantifies the analyticity of a
function $g$ of the phase space:
\begin{equation}
\label{analytic}
\|g\|_{\mu}= \sup_{n,\xi}\e^{\mu \left \langle n,\xi \right \rangle}|\widehat{g}_n(\xi)|,
\end{equation}
where $\mu>0$ is a parameter and
$\left \langle n,\xi \right \rangle= (1+n^2+\xi^2)^{\frac{1}{2}}$.

To take into account the decay of the analytic regularity,
we define the weighted-in-time analytic norm of the solution $h^T(t,x,v)$
as 
\begin{equation}
\label{CK}
N_{\lambda, T}[h^T]= \sup_{(\mu,t) \in D_{\lambda,T}} \alpha_\delta^T(\mu,t)^{1/2}\|h^T(t)\|_{\mu},
\end{equation}
where
\begin{equation}
\label{domain}
D_{\lambda, T}=\{ (\mu, t)\in
[0,\lambda)\times[0, T], \alpha_\delta^T(\mu, t)>0
\}
\end{equation}
and $\alpha_\delta^T(\mu, t)=\lambda - \mu -  a_{T, \delta}(t)$.
The function $a_{T, \delta}(t)$ is the unique solution of the following
ordinary differential equation
\begin{equation}
\label{eqdiff}
\begin{cases}
\dot a_{T, \delta}(t)=-\delta \e^{-a_{T, \delta}(t)t}(1+t)  & \text{if} \quad 0 \le t \le T \\
a_{T,\delta}(T)=0,
\end{cases}
\end{equation}
and measures the loss of analytic regularity of the solutions
with respect to the final datum, as in \eqref{switchDelta} below:
it is $0$ at time $T$, and it is maximum at $t=0$.
In view of the limit $T\to +\infty$, we need the
following lemma, proved in the appendix.
\begin{lemma}
  \label{lemma:a}
  For $\delta>0$
  the unique solution of the backward 
  Cauchy problem \eqref{eqdiff} is positive and
  decreasing in time, and verifies 
\begin{equation*}
 a_{T,\delta}(0) \le C(\delta),
\end{equation*}
with $C(\delta) \to 0$ when $\delta$ goes to zero.
The solution $a_{\infty, \delta}(t)$ with initial datum
$$a_{\infty,\delta}(0) = \lim_{T \to +\infty}a_{T,\delta}(0)$$
is positive in $[0,+\infty)$ and 
\[
\lim_{t \to +\infty} a_{\infty,\delta}(t)=0.
\]
As a consequence, given $\lambda>0$,
we can choose $\delta$ sufficiently small such
that there exist $\mu\in (0,\lambda)$ for which
for any $T>0$,
$[0,\mu]\times [0,T] \subset D_{\lambda,T}$.
\end{lemma}

We define  $\mathcal{B}_{\lambda, T}$
the space of function $h(t,x,v)$, defined for $t\in  [0,T]$,
with $N_{\lambda, T}[h]< +\infty$,
and $\mathcal{B}_{\lambda, \infty}$ as the
space of functions $h(t,x,v)$ with $t\in [0,+\infty)$ such that
$N_{\lambda, \infty}[h]<+\infty$, where
$N_{\lambda, \infty}[h]$ is defined in the region
$D_{\lambda, \infty}=\{ (\mu, t)\in [0,\lambda)\times[0, +\infty),
\alpha_\delta^\infty(\mu, t)>0 \}$ with
$\alpha_\delta^\infty(\mu, t)=\lambda - \mu - a_{\infty, \delta}(t)$.

\subsection{Estimates for $\zeta^T$}

As we show more accurately in the following lemma,
eq. \eqref{zetaequa} for the field $\zeta^T$
has the structure of a Volterra equation.
In order to invert the term of order one in the equation, we use the following
classical result about
the theory of Volterra operators.

\begin{Theorem}[\cite{GLS}, p. 45]
  \label{resolvo}
  Given a Volterra equation of the form
  $f(t) + j*f(t) = g(t)$,
  where
  \[
    j*f(t)=\int_0^t j(t-s)f(s) \, ds
  \]
  with $j \in L^1(\mathbb{R}_+)$.  
  The resolvent kernel $r$,
  i.e. the unique solution of the equation
  $$
  r+j * r = j,
  $$
  belongs to $L^1(\mathbb{R}_+)$ if and only if 
  $$
  \mathcal{L}[j](\sigma) \neq -1 \quad \text{for} \quad \Re \sigma \ge 0,
  $$
  where 
  $$\mathcal{L}[j](\sigma)=\int_0^{+\infty} \e^{-\sigma t} j(t) \, dt$$
  is the Laplace transform of $K$.
  The solution $f$ is then given by $f(t)=g(t)-r*g(t)$.
\end{Theorem}
\vskip.3cm
We can now state the inversion lemma. We set
\begin{equation}
  \label{kappa}
  j_n(t)\equiv \imm \frac{n}{2}\widetilde{\eta'}(nt),
\end{equation}
and
\begin{equation}
\label{effeeps}
H^T_{\eps}(t)=\widehat{h^T}_ n(T,nt)
-\frac{\eps}{2} \sum_{k= \pm 1}
\int_t^T \zeta^T_k(s)\widehat{h^T}_{n-k}(s, nt-ks)kn(s-t) \, ds.
\end{equation}
\begin{lemma}
\label{inversione}
Let $\lambda>0$ with $\|h_\infty\|_{\lambda}< +\infty$ and
$\| \eta \|_{\lambda}< +\infty$. 
Assume that
$$
\mathcal{L}[j_1](\sigma) \neq 1, \quad \Re \sigma \ge 0
$$
then
\[
  M_{\lambda, T}[\zeta^T] \le C_\lambda M_{\lambda, T}[H^T_\eps].
\]
\end{lemma}
We notice that the condition on the Laplace transform
is fulfilled also by $j_{-1}$
since $\overline {j_1}=j_{-1}$.
\begin{proof}
Let us define 
$
\phi_\lambda(t)=\e^{\lambda(T-t)} \zeta^T_1(T-t)$,
$F_{\eps}(t)=\e^{\lambda (T-t)} H^T_{\eps}(T-t).
$
Multiplying by $\e^{\lambda t}$, \eqref{zetaequa} can be rewritten as
\begin{equation}
\label{convo}
\phi_\lambda(t)+\jmath_\lambda*  \phi_\lambda(t)= F_{\eps}(t),
\end{equation}
for $t\in [0,T]$, where $\jmath_\lambda(t)=-\e^{-\lambda t}j_{-1}(t)$.
We notice that $\jmath_\lambda \in L^1(\mathbb{R}_+)$
and if $\Re \sigma \ge 0$
$$
\mathcal{L}[\jmath_\lambda](\sigma)=-\mathcal{L}[j_{-1}](\sigma+\lambda)\neq -1.
$$
Then, from Theorem \eqref{resolvo}, the resolvent kernel
$r_{\lambda}$ related to $\jmath_\lambda$ belongs
to $L^1(\mathbb{R}_+)$.
Convolving with $r_{\lambda}$ in \eqref{convo}, we get
$$
\phi_\lambda(t)= F_\eps(t)- \int_0^t r_{\lambda}(t-s)F_{\eps}(s)\, ds.
$$
Taking the absolute values, it holds
$$
M_{\lambda, T}[\zeta^T]=\sup_{t \in [0, T]} |\phi_\lambda(t)| \le M_{\lambda, T}[H^T_\eps]
+  \|r_{\lambda}\|_{L^1(\mathbb{R}_+)} M_{\lambda, T}[H^T_\eps]
$$
and the thesis follow with $C_\lambda=1+
\|r_{\lambda}\|_{L^1(\mathbb{R}_+)}$.
\end{proof}

We now state the main estimate of this section.

\begin{Proposition}
\label{stimadelcampo}
Let $\zeta^T_{\pm 1}$ solution of \eqref{zetaequa} and suppose
$N_{\lambda, T}[h^T]<+\infty$. Then, under the hypothesis of Lemma
\eqref{inversione}, we have
\begin{equation}
\label{stima1}
M_{\lambda, T}[\zeta^T] \le C_\lambda \| h_\infty \|_{\lambda}+  \eps \frac{C_\lambda}{\lambda^2 \sqrt{\lambda-a_{\infty, \delta}(0)}}M_{\lambda, T}[\zeta^T]N_{\lambda, T}[h^T].
\end{equation}
\end{Proposition}

\begin{proof}
  From
  Lemma \eqref{inversione} we
  need only to estimate $N_{\lambda, T}[H^T_\eps]$.
  Being $h^T(T,x,v)=h_\infty(x,v)$, from \eqref{effeeps} we have
  \begin{equation}
    \label{eqdastimare}
	\begin{split}
    \e^{\lambda t}|H^T_{\eps}(t)|
    &\le \| h_{\infty} \|_{\lambda}\\
&+\eps M_{\lambda, T}[\zeta^T]N_{\lambda, T}[h^T] \sum_{k= \pm 1}\int_t^T\frac{ \e^{-\lambda(s-t)-\mu'\left \langle n-k, nt-ks \right \rangle}}
    {\alpha^T(\mu',s)^{1/2}}
   (s-t) \, ds,
\end{split}
  \end{equation}
  for any $\mu'<\lambda - a_{T, \delta}(s)$. 
  Then, by choosing $\mu' = 0$, and 
  using that $a_{T,\delta}(s)\le
  a_{T,\delta}(0) <   a_{\infty,\delta}(0)$
  we get
  \[
  \e^{\lambda t}|H^T_{\eps}(t)|\le \| h_\infty \|_{\lambda}+
  \eps\frac{M_{\lambda, T}[\zeta^T]N_{\lambda, T}[h^T]}
  {(\lambda-a_{\infty, \delta}(0))^{1/2}} \int_t^T
  \e^{-\lambda(s-t)}(s-t) \, ds.
\]
\end{proof}

\subsection{Estimates for $h^T$}
Now we turn to give a Cauchy-Kovalevskaya estimate on $h^T$.
Due to the loss of analytic regularity in time, it is crucial to use the weighted-norm introduced in \eqref{CK}.
\begin{Proposition}
\label{due}
Let $h^T$ a solution of \eqref{tempoT} and assume $M_{\lambda, T}[\zeta^T]<
+\infty$ then the following estimate holds:
\begin{equation}
\label{stima2}
N_{\lambda, T}[h^T]\le C\|h_\infty\|_{\lambda}+\frac{C}{\delta}M_{\lambda, T}[\zeta^T]\|\eta\|_\lambda+\eps \frac{C}{\delta}M_{\lambda, T}[\zeta^T]N_{\lambda, T}[h^T].
\end{equation}
\end{Proposition}

\begin{proof}
Fixing  $\mu<\lambda- a_{T,\delta}(t)$, from \eqref{differenziale} we get 
\begin{equation}
\label{lungo}
e^{\mu \left \langle n, \xi \right \rangle} |\widehat{h^T}_n(t,\xi)| \le\|h_\infty\|_{\lambda} +\e^{\mu  \left \langle n, \xi \right \rangle}|D^T_n(t,\xi)|+\e^{\mu \left \langle n, \xi \right \rangle}|E^T_n(t,\xi)|
\end{equation}
where
\begin{equation*}
D^T_n(t,\xi)=\delta_{n,\pm 1}\frac{\imm}{2}n \int_t^T \zeta^T_n(s) \widetilde{\eta'}(\xi-ns) \, ds
\end{equation*}
and
\begin{equation*}
E^T_n(t,\xi)\equiv \frac{\eps}{2} \sum_{k= \pm 1}\int_t^T \zeta^T_k(s)\widehat{h^T}_{n-k}
  (s,\xi-ks)k(\xi-ns) \, ds.
\end{equation*}
We estimate separately the two terms.
As regards  $E_n$,  since
\begin{equation*}
  \e^{\mu\left \langle n, \xi \right \rangle} \le
  \e^{\mu \left \langle n-k, \xi - ks\right \rangle}
  \e^{\mu \left \langle k, ks \right \rangle},
\end{equation*}
by the triangular inequality and taking
$$ \mu(s)=\frac{\lambda+\mu - a_{T,\delta}(s)}{2},$$
{\it i.e.} the middle point between $\mu$ and $\lambda - a_{T, \delta}(s)$,
we have
\begin{equation*}
\begin{split}
\e^{\mu \left \langle n, \xi \right \rangle}|E^T_n(t,\xi)&|\le\sum_{k=\pm1} M_{\lambda, T}[\zeta^T]\times\\
&\times \int_t^T\e^{-(\lambda-\mu) s}\|h^T(s)\|_{\mu(s)} \e^{- (\mu(s)-\mu)\left \langle n-k, \xi-ks \right \rangle} |\xi-ns|\, ds, 
\end{split}
\end{equation*}
where we have also used that $\left \langle \pm 1 , \pm s \right \rangle
\le C + s$.
Noting that
\begin{equation*}
\e^{-(\mu(s)-\mu)\left \langle n-k,\xi-ks\right \rangle} |\xi-ns|\le\frac{2(1+s)}{\lambda -\mu- a_{T, \delta}(s)},
\end{equation*}
we get
\begin{equation}
\label{switchDelta}
\e^{\mu \left \langle n, \xi \right \rangle}|E^T_n(t,\xi)| \le \frac{C}{\delta}M_{\lambda, T}[\zeta^T]N_{\lambda, T}[h^T] \int_t^T   \frac{\delta \e^{-(\lambda-\mu)s}(1+s)}{\alpha^T(\mu,s)^{3/2}} \, ds.
\end{equation}
Being $\lambda-\mu>a_{T,\delta}(s)$ and using the definition of $a_{T, \delta}$
in \eqref{eqdiff}
$$
\frac {\e^{-(\lambda-\mu)s}(1+s)}{\alpha^T(\mu,s)^{3/2}} \le -\frac 2{\delta}
\frac{\de}{\de s} \alpha^T(\mu,s)^{-1/2}
$$
and then
$$
\e^{\mu \left \langle n, \xi \right \rangle}|E_n(t,\xi)| \le \frac{C}{\delta}
M_{\lambda, T}[\zeta^T] N_{\lambda, T}[h^T] \frac{1}{\alpha^T(\mu,t)^{1/2}}.
$$
As regards $D^T_n$,
for $\mu<\lambda-a_{T,\delta}(t)$,
\begin{equation*}
\begin{split}
\e^{\mu \left \langle n, \xi \right \rangle}|D^T_n(t,\xi)|&\le  C M_{\lambda, T}[\zeta^T]\|\eta\|_\lambda \int_t^T \e^{-(\lambda-\mu) s}\e^{-(\lambda-\mu) \left \langle \xi-ns \right \rangle}\left \langle \xi-ns \right \rangle\,ds \\
&\le  \frac{C}{\delta} M_{\lambda, T}[\zeta^T] \| \eta\|_{\lambda}\int_t^T \frac{\delta\e^{-a_{T,\delta}(s) s}(1+s)}{\alpha^T(\mu,s)}\, ds
\end{split}
\end{equation*}
where in the last inequality we have used that $\lambda-\mu > 
\lambda-\mu - a_{T,\delta}(s)= \alpha^T(\mu, s)$
and also that $\lambda-\mu>a_{T,\delta}(s)$.
Computing the integral,
we get
$$
\e^{\mu \left \langle n, \xi \right \rangle}|D^T_n(t,\xi)|\le \frac{C}{\delta} M_{\lambda, T}[\zeta^T] \| \eta\|_{\lambda} \ln \Bigl{(} \frac{\alpha^T(\mu,T)}{\alpha^T(\mu,t)} \Bigr{)}.
$$
We conclude the proof multiplying \eqref{lungo} by $\alpha^T(\mu,t)^{1/2}$,
and taking the supremum over $D_{\lambda, T}$.
\end{proof}

\subsection{The backward result}
\begin{Theorem}
\label{MAIN}
Let $h_\infty\in L^1(S^1 \times \R)$ analytic such that $\|h_\infty\|_{\lambda}<+\infty$ with $\lambda>0$ . Consider $\eta \in L^1(\mathbb{R})$ analytic such that $\|\eta\|_\lambda<+\infty$. Moreover, assume
$$
\mathcal{L}[j_1](\sigma) \neq 1, \quad \Re \sigma \ge 0,
$$
with $j_1$ as in \eqref{kappa}.
Then, for small values of $\eps$, there exists a unique solution $h(t,x,v)$ of \eqref{principale} with $N_{\lambda, \infty}[h]<+\infty$ such that 
\begin{equation*}
\lim_{t \rightarrow +\infty}\| h(t,x,v)-h_\infty(x,v)\|_{\infty}=0
\end{equation*}
with exponential rate.
\end{Theorem}

\begin{proof}
For every $T$ 
we get the unique solution $h^{T}$ of \eqref{tempoT} using the
following iterative procedure. For $j \ge 0$ and $0 \le t \le T$ let
\begin{equation}
\begin{split}
\label{iterative1}
\partial_t \widehat{h}^{(j+1), T}_n(t, \xi)&=\delta_{n, \pm 1}n\frac{\imm }{2}\zeta^{(j), T}_n(t)\widetilde{\eta'}(\xi-nt)\\
&- \eps  \sum_{k= \pm 1}k\frac{\zeta^{(j), T}_k(t)}{2}\widehat{h}^{(j+1), T}_{n-k}(t, \xi-kt)(\xi -nt),
\end{split}
\end{equation}
where $\zeta^{(j),T}_1(t)$ is defined by
\begin{equation*}
\begin{split}
\zeta^{(j), T}_1(t)&=\widehat{h}_{\infty}(1,t)-\frac{\imm}{2} \int_t^T \zeta^{(j), T}_1(s)\widetilde{\eta'}(t-s)\, ds\\
&- \frac{\eps}{2} \sum_{k= \pm 1}\int_t^T \zeta^{(j), T}_k(s)\widehat{h}^{(j), T}_{n-k}(s, t-ks)k(s-t) \, ds, \\
\end{split}
\end{equation*}
where $\zeta^{(j),T}_{-1}=\widebar{\zeta_1^{(j),T}}$ and with initial step $ h^{(0), T}(t,x,v)=h_\infty(x,v)$. \\
Then $h^{(j),T}$ verifies the same bounds of the {\it a priori} estimates in \eqref{stima1} and \eqref{stima2}:
\begin{equation*}
M_{\lambda, T}[\zeta^{(j), T}] \le C\| h_\infty \|_{\lambda}+  \eps C M_{\lambda, T}[\zeta^{(j), T}]N_{\lambda, T}[h^{(j), T}]
\end{equation*}
and
\begin{equation*}
\begin{split}
N_{\lambda, T}[&h^{(j+1), T}]\le C\|h_\infty\|_{\lambda}+CM_{\lambda, T}[\zeta^{(j), T}]\Bigl{(}\|\eta\|_\lambda+\eps N_{\lambda, T}[h^{(j+1), T}]\Bigr{)}\\
& \le C\| h_\infty \|_{\lambda}+\eps C M_{\lambda, T}[\zeta^{(j),T}]\Bigl{(} N_{\lambda, T}[h^{(j),T}]+ N_{\lambda, T}[h^{(j+1),T}] \Bigr{)},
\end{split}
\end{equation*}
where we have used \eqref{stima1} in the last inequality and where $C$ is a
generic constant depending on $\lambda$ and $\delta$.
Since $N_{\lambda, T}[h^{(0),T}] \le C \|h_\infty\|_\lambda$, taking $\eps  \|h_\infty\|_\lambda$ sufficiently small, we get that $M_{\lambda, T}[\zeta^{(j), T}]$ and $N_{\lambda, T}[h^{(j+1), T}]$ are uniformly bounded in $j\ge 0$. 
Then, taking $\delta'>\delta$ in \eqref{lemma:a}, the time derivative of $h^{(j),T}$ is uniformly bounded in $N_{\lambda,T}[\cdot]$. Hence there exists a subsequence $h^{(j_k), T}$ which converge to a function $h^T \in \mathcal{B}_{\lambda, T}$, while $\zeta_{\pm 1}^{(j_k), T}$ converge to a function $\zeta_{\pm 1}^T$ such that $M_{\lambda, T}[\zeta^T]< +\infty$.
Then $h^T_n(t,nt)=\zeta^{T}_n(t)$ for $n=\pm 1$ and it is a solution of the nonlinear problem \eqref{tempoT}. \\
We now extend  $h^{T}(t,x,v)=h_\infty(x,v)$ for $t \ge T$ and we consider the sequence of solutions $\{h^T\}$, with $h^T \in \mathcal{B}_{\lambda, \infty}$. We can see that $h^T$ fulfills the Cauchy property as a function of $T$ in $\mathcal{B}_{\lambda',\infty}$ with $\lambda>\lambda'>a_{\delta, \infty}(0)$. In fact, fixed $T^*$, taking $T'\ge T \ge T^*$, we have for $t\le T$
\begin{equation*}
\begin{split}
&\widehat{h^{T'}}_n(t, \xi) - \widehat{h^T}_n(t, \xi)=\delta_{n, \pm 1}n\frac{\imm}{2}\int_t^T \Bigl{(} \zeta^T_n(s)-\zeta^{T'}_n(s) \Bigr{)}\widetilde{\eta'}(\xi-ns)\, ds\\
& - \eps  \sum_{k= \pm 1}k \int_t^T \frac{\Bigl{(} \zeta^T_k(s)-\zeta^{T'}_k(s)\Bigr{)}}{2}\widehat{h^T}_{n-k}(s, \xi-ks)(\xi -ns)\, ds\\
& - \eps  \sum_{k= \pm 1}k \int_t^T \frac{\zeta^{T'}_k(s)}{2}\Bigl{(} \widehat{h^T}_{n-k}(s, \xi-ks)-\widehat{h^{T'}}_{n-k}(s, \xi-ks)\Bigr{)}(\xi -ns)\,ds\\
& + \delta_{n, \pm 1}n\frac{\imm}{2}\int_T^{T'} \zeta^{T'}_n(s)\widetilde{\eta'}(\xi-ns)\, ds \\
& + \eps  \sum_{k= \pm 1}k \int_T^{T'} \frac{\zeta^{T'}_k(s)}{2}\widehat{h^{T'}}_{n-k}(s, \xi-ks)(\xi -ns)\,ds
\end{split}
\end{equation*}
and an analogous of equation \eqref{zetaequa} holds
for $\zeta^T-\zeta^{T'}$. Doing  estimates
in the style of \eqref{stima1} and \eqref{stima2}, we get
\begin{equation*}
M_{\lambda', T}[\zeta^{T'}-\zeta^{T}] \le\eps C M_{\lambda', T}[\zeta^{T'}-\zeta^{T}]+\eps CN_{\lambda', \infty}[h^{T'}-h^{T}]+\eps \frac{C}{{\lambda'}^2} \e^{-(\lambda-\lambda')T^{*}}
\end{equation*}
and
\begin{equation}
\label{stimacauchy}
\begin{split}
N_{\lambda', \infty}[h^{T'}-h^{T}]&\le C M_{\lambda', T}[\zeta^{T'}-\zeta^{T}]+\eps C  M_{\lambda', T}[\zeta^{T'}-\zeta^{T}]\\
&+\eps C N_{\lambda', \infty}[h^{T'}-h^{T}]+\frac{(1+\eps)C}{\min\{1,\lambda-\lambda'\}^3}\e^{-\frac{(\lambda-\lambda')}{2}T^{*}}.
\end{split}
\end{equation}
Hence, using again the smallness of $\eps$, we conclude that
\[
  \lim_{T^*\to +\infty}\sup_{T'\ge T \ge T^*} N_{\lambda', \infty}[h^{T'}-h^{T}] = 0.
\]
Being uniformly bounded in $\mathcal{B}_{\lambda, \infty}$, the sequence $\{h^T\}$ converge to a function $h \in \mathcal{B}_{\lambda, \infty}$
and, passing to the limit by dominated convergence in the integral formulation, $h(t,x,v)$ is solution of the nonlinear equation \eqref{principale} in
$[0,+\infty)$.
So, taking $\widebar{\mu}<\lambda- a_{\infty,\delta}(0)$, we have that $\|h(t,x,v)-h_{\infty}(x,v)\|_{\widebar{\mu}} \to 0$.\\
We get the uniqueness of the solutions with a similar procedure.
Let $g(t,x,v)$ and $h(t,x,v)$
be two solutions of \eqref{principale} with the same asymptotic datum
$h_\infty$. Proceeding as before, we can prove that
they verify the estimates \eqref{stima1} and \eqref{stima2}.
Hence, denoting $\zeta_h$ the electric field associated to $h$, we get
$$
\max \Bigl{(} {N_{\lambda,\infty}[h], M_{\lambda,\infty}[\zeta_h]}
\Bigr{)}\le C \|h_\infty\|_\lambda
$$ 
and analogously for $g(t,x,v)$.
Estimating $N_{\lambda, \infty}[g-h]$, we obtain the same estimates as in \eqref{stimacauchy} without the rest terms:
$$
A\equiv \max \Bigl{(} {N_{\lambda,\infty}[g-h], M_{\lambda,\infty}[\zeta_g-\zeta_h]} \Bigr{)}\le C(\eps) A.
$$ 
Using the smallness on $\eps$ as before, we have $C(\eps)<1$,
from which the uniqueness follows.

We remark that in \cite{CM}, 
in the case of the scattering problem for the Vlasov-Poisson equation,
the uniqueness is guaranteed 
for a wider class of solutions,
not necessarily analytic.
\end{proof}

\section{Non-perturbative regime}

Using the backward approach for large times it is possible to construct solutions without perturbating around the homogeneous equilibrium $\eta(v)$, in the style of \cite{CM}. The price to pay is that the
analytic estimates hold only in $[\tau,+\infty)$ for  
$\tau$ large enough.

Fixed  an analytic asymptotic state $\omega(x,v)$, consider \eqref{HMF}
and write
\[
f(t,x,v)=\widebar{\omega} (v)+g(t,x,v),
\] 
where $\widebar{\omega}$ is the mean of $\omega(x,v)$ with respect to the $x$ variable. 
Then $h(t,x,v)=g(t,x+vt,v)$ verifies the equation
\begin{equation*}
\partial_t h=\{ \psi[h],\widebar{\omega}+ h \}
\end{equation*}
where $\psi$ is defined as in \eqref{campo}.
For $T\ge\tau$, let us consider the following sequence of problems
\begin{equation*}
\begin{cases}
\partial_t h^T=\{ \psi[h^T],\widebar{\omega}+ h^T \} \quad \tau \le t\le T,\\
h^T(T,x,v)=(\omega-\widebar{\omega})(x,v).
\end{cases}
\end{equation*}
We introduce the weighted norm
\begin{equation*}
Q_{\lambda, T}[h^T]=\sup_{(\mu,t) \in \Omega_{\lambda,T}} \theta^T(\mu,t)^{1/2}\|h^T(t)\|_\mu,
\end{equation*}
with the weight $\theta^T(\mu,t)=(\lambda - \mu -\Delta a_T(s))$, where $\Delta=\lambda'/a_{\infty}(\tau)$, $\lambda'<\lambda$ and $a_T(s)$ is defined as in \eqref{eqdiff} putting $\delta=1$. Notice now that $\Delta$ is a diverging quantity for sufficiently large $\tau$. 
Here $\Omega_{\lambda,T}=\{ (\mu,t) \in [0,\lambda)\times [\tau,T], \theta^T(\mu,t)>0$\} and, as in the previous case, we can give the analogous definitions for $Q_{\lambda, \infty}[\cdot]$, $\theta^{\infty}$ and $\Omega_{\lambda, \infty}$.\\
We define $\zeta^T_n(t)=\widehat{h^T}_n(t, nt)$, $n=\pm 1$, then $\zeta^T$ verifies the following equation:
\begin{equation}
\label{variata}
\zeta^T_n(t)=\int_t^T \zeta^T_n(s) j_n(t-s)\, ds+W^T(t),
\end{equation}
where we have defined
\[
W^T(t)\equiv \widehat{\omega}_n(nT) - \frac{1}{2} \sum_{k= \pm 1}\int_t^T \zeta^T_k(s)\widehat{h^T}_{n-k}(s, nt-ks)kn(t-s) \, ds
\]
and 
\begin{equation}
\label{kappaomega}
j_n(t)=\imm \frac{n}{2} \widetilde{\omega'}_0(nt).
\end{equation}
As in \eqref{campoindietro} we denote 
\[
P_{\lambda, T}[\zeta^T]=\sup_{t \in [\tau, T]}\e^{\lambda t} |\zeta^T_1(t)|=\sup_{t \in [\tau, T]}\e^{\lambda t} |\zeta^T_{-1}(t)|.
\]
We can now state the following theorem.

\begin{Theorem}
\label{sperturbo}
Let $\omega\in L^1(S^1 \times \R)$ analytic such that $\|\omega\|_{\lambda}<+\infty$ and assume that 
\begin{equation}
\label{qinverso}
\mathcal{L}[j_1](\sigma) \neq 1, \quad \Re \sigma \ge 0,
\end{equation}
with $j_1$ as in \eqref{kappaomega}. Then,
for sufficiently large $\tau$, there exists a unique solution $h(x,v,t)$ of
\[
  \partial_t h=\{ \psi[h],\widebar{\omega}+h\} \quad \text{if} \quad \tau \le t
  < +\infty,\\
\]
 with $Q_{\lambda,\infty}[h]< +\infty$  such that 
\begin{equation*}
\lim_{t \rightarrow +\infty}\| h(t,x,v)-(\omega-\widebar{\omega})(x,v)\|_{\infty}=0
\end{equation*}
with exponential rate.
\end{Theorem}

\begin{proof}[Proof of theorem \eqref{sperturbo}]
The proof goes in the same way of \eqref{MAIN} but instead of using the smallness of $\eps$, we can use the size of $\Delta$.
Indeed as in Proposition \eqref{due} we can estimate $h^T$ in $[\tau, T]$ where $h^T$ verifies the equation
\begin{equation*}
\widehat{h^T}_n(t,\xi)=D_n(t,\xi)- \sum_{k= \pm 1}\int_0^t k\frac{\zeta^T_k(s)}{2}\widehat{h^T}_{n-k}(s, \xi-ks)(\xi -ns)\, ds
\end{equation*}
with
$$
D_n(t,\xi)=\delta_{n, \pm 1} \frac{\imm}{2}n \int_t^T \zeta^T_n(s)\widetilde{\omega'_0}(\xi-ns) \, ds.
$$
We first treat the case $n\neq \pm 1$. As in \eqref{switchDelta} and using $\lambda-\mu>\Delta a_T(s)>a_T(s)$ we have
\begin{equation*}
\begin{split}
e^{\mu \left \langle n, \xi \right \rangle} |\widehat{h^T}_n(t,\xi)| &\le \|\omega\|_{\lambda}+C P_{\lambda, T}[\zeta^T] Q_{\lambda, T}[h^T]\int_t^T   \frac{\e^{-(\lambda-\mu)s}(1+s)}{\Theta^T(\mu,s)^{3/2}} \, ds \\
&\le \|\omega\|_{\lambda}+ C\frac{ P_{\lambda, T}[\zeta^T] Q_{\lambda, T}[h^T]}{\Delta}\int_t^T   \frac{\Delta\e^{-a_T(s)s}(1+s)}{\Theta^T(\mu,s)^{3/2}} \, ds
\end{split}
\end{equation*}
and thus, since
$$
\frac{\de}{\de t} \Theta^T(\mu,t)^{-1/2}=-\frac{\Delta}{2} \frac{\e^{-a_T(s)s}(1+s)}{\Theta^T(\mu, t)^{3/2}}
$$
we get
\begin{equation}
\label{log1}
\begin{split}
e^{\mu \left \langle n, \xi \right \rangle} |\widehat{h^T}_n(t,\xi)| \le \|\omega\|_{\lambda}+ C\frac{ P_{\lambda, T}[\zeta^T] Q_{\lambda, T}[h^T]}{\Delta \Theta^T(\mu,s)^{1/2}}.
\end{split}
\end{equation}
Now we estimate
$D_n(t,\xi)$, $n=\pm 1$. Take $\mu<\lambda-\Delta a_T(t)$, hence $\lambda-\mu>(\lambda-\mu-\Delta a_T(s))/2$, so we get
\begin{equation}
\label{log2}
\begin{split}
e^{\mu \left \langle n, \xi \right \rangle}|D_n(t,\xi)|& \le C P_{\lambda, T}[\zeta^T]\|\omega\|_\lambda \int_t^T \e^{-(\lambda - \mu)s}\e^{-(\lambda-\mu)\left \langle \xi-ns \right \rangle}\left \langle \xi-ns \right \rangle \, ds \\
&\le C P_{\lambda, T}[\zeta^T]\|\omega\|_\lambda \int_t^T \frac{\e^{-a_T(s)s}(1+s)}{ \Theta^T(\mu,s)}ds\\
&\le  C \frac{P_{\lambda, T}[\zeta^T]\|\omega\|_\lambda}{\Delta}\ln\Bigl{(}\frac{\Theta^T(\mu,T)}{\Theta^T(\mu,t)}\Bigr{)}. 
\end{split}
\end{equation}
Hence, multiplying by $\Theta^T(\mu,t)^{1/2}$ in \eqref{log1} and \eqref{log2} we get
\[
Q_{\lambda, T}[h^T] \le C\|\omega\|_{\lambda}+\frac{C}{\Delta}P_{\lambda, T}[\zeta^T]\|\omega\|_\lambda+\frac{C}{\Delta}  P_{\lambda, T}[\zeta^T] Q_{\lambda, T}[h^T].
\]
Regarding $\zeta^T$ in \eqref{variata}, using \eqref{qinverso} and \eqref{resolvo} we have
\begin{equation*}
P_{\lambda, T}[\zeta^T ]\le C_\lambda P_{\lambda, T}[W^T].
\end{equation*}
We need better estimates than that in \eqref{eqdastimare}.
We get them by splitting the two modes $k=\pm1$ in 
\begin{equation}
\label{dasplittare}
\sum_{k=\pm1} \int_t^T \zeta^T_k(s) \widehat{h^T}_{1-k}(s,t-ks)k(t-s) \, ds=B_1+B_{-1}.
\end{equation}
 If $k=-1$, for $\mu'<\lambda-\Delta a_{\infty}(\tau)=\lambda-\lambda'$, we get
\begin{align*}
\e^{\lambda t}|B_{-1}|&\le P_{\lambda, T}[\zeta^T] \int_t^T \e^{-\lambda(s-t)}\frac{Q_{\lambda,T}[h^T]}{\Theta(\mu',s)^{1/2}}\e^{-\mu'(t+s)}(s-t) \, ds \\
&\le P_{\lambda, T}[\zeta^T]Q_{\lambda,T}[h^T] \frac{\e^{-2\mu' \tau}}{(\lambda-\mu'- \lambda')^{1/2}}\int_t^T \e^{-\lambda(s-t)}(s-t) \, ds \\
&\le C P_{\lambda, T}[\zeta^T] Q_{\lambda,T}[h^T]\frac{\sqrt{\tau}}{\lambda^2}\e^{-(\lambda-\lambda')\tau},
\end{align*}
where we have taken the infimum on $\mu' \in [0, \lambda-\lambda']$ in the last inequality.
In the other case, using that $\omega-\widebar{\omega}$ has mean zero in the $x$ variable, we have
\begin{equation}
\label{accazero}
\widehat{h^T}_0(s,t-s)=\sum_{k=\pm1} \int_s^T \zeta^T_k(l) \widehat{h^T}_{-k}(l,t-s-kl)k(t-s)\, dl.
\end{equation}
Replacing \eqref{accazero} in \eqref{dasplittare} we obtain
\begin{equation*}
\begin{split}
\e^{\lambda t}|B_{1}|&\le P_{\lambda, T}[\zeta^T] \int_t^T \e^{-\lambda(s-t)}(s-t)|\widehat{h^T}_0(s,t-s)| \, ds\\
&\le C P_{\lambda, T}[\zeta^T]\frac{Q_{\lambda, T}[h^T]}{\lambda-\lambda'} \int_t^T \e^{-\lambda(s-t)}(s-t)^2 \int_s^T \e^{-\lambda l} dl \\
&\le C\frac{ P_{\lambda, T}[\zeta^T]Q_{\lambda, T}[h^T]}{\lambda^3(\lambda-\lambda')}\e^{-\lambda \tau}.
\end{split}
\end{equation*}
Hence 
\begin{equation*}
P_{\lambda, T}[\zeta^T] \le\| \omega \|_{\lambda}+ CP_{\lambda, T}[\zeta^T] Q_{\lambda, T}[h^T]\Biggl{(}\frac{\sqrt{\tau}}{\lambda^2}\e^{-(\lambda-\lambda')\tau} +  \frac{\e^{-\lambda \tau}P_{\lambda, T}[\zeta^T] }{\lambda^3(\lambda-\lambda')}\Biggr{)}
\end{equation*}
and we can reason as in the proof of the main theorem avoiding to use the smallness of $\eps$.
\end{proof}

\begin{Remark}
  We notice that in this setting we have obtained an Eulerian
  analog of the scattering result in \cite{CM}, in the special case
  of the HMF model.
  In \cite{CM} Caglioti and Maffei,
  using the Lagrangian description of the flow, 
  obtain the damping result for the Vlasov-Poisson equation,
  by a fixed point technique, considering
  an asymptotic state $\omega$ with $\|\omega\|_\lambda<+\infty$  such that
  \[
    \omega(x,v)\le \frac{M}{(1+v^2)^2}
  \]
  for some $M>0$ and $\lambda \ge C\sqrt{M}$, with $C$ some purely numerical constant. Here we show that such class of final data fulfills condition \eqref{qinverso}, if $\lambda > \pi\sqrt{M}$.
Indeed, taking $n=1$ in \eqref{variata}
and multiplying by $\e^{\lambda t}$ we get as in \eqref{convo}
\[
\phi^T_\lambda(t)+\phi^T_\lambda*\jmath_\lambda(t)=\e^{\lambda (T-t)} W^T(T-t)
\]
with $\jmath_\lambda(t)=-\e^{-\lambda t} j_{-1}(t)$ and $\phi^T_\lambda(t)=\e^{\lambda (T-t)} \zeta^T_1(T-t)$.
So it is sufficient to notice that, since $|\widetilde{\omega_0}|\le M \pi^2$, we have
\[
|\mathcal{L}[\jmath_{\pm 1}](\sigma)| \le M \pi^2 \int_0^{+\infty} \e^{-\Re \sigma t} \e^{-\lambda t} t \, dt\le \pi^2 \frac{M}{\lambda^2}<1, \quad \Re \sigma \ge 0
\]
hence \eqref{resolvo} holds. 
\end{Remark}
\begin{Remark}
  The non-perturbative scattering result in theorem \eqref{sperturbo} allows the choice of asymptotic states $\omega$ within a distance of $O(1)$ from a given homogenous state $\eta(v)$. This fact poses a significant difference with respect to the forward perturbative results where, as we show in Section 4,
given an equilibrium $\eta(v)$ which verify some stability properties, there exists an $\eps_0>0$ such that every initial data in an analytic neighborhood of $\eta$ of $O(\eps)$ with $\eps<\eps_0$ verifies the Landau damping.\\
Actually, solutions of the backward and forward problems are of a different type. Indeed, in the case of the attractive HMF model \footnote{Except this paragraph, the choice of an attractive or repulsive potential is indifferent in this work.}, it is easy to find  non-homogeneous BGK stationary solutions $\omega(x,v)$  of the HMF that can be chosen as scattering asymptotic datum for the HMF, i.e. such that there exists a solution $f_\omega(x,v,t)$ such that
\[
\lim_{t \to +\infty}\|f_\omega(t,x,v) - \omega(x-v t,v))\|_{\infty}=0.
\]
This solution $f_\omega$ could never be a Landau Damping solution because it is not close, in a strong norm, say $L_1,$ to its weak asymptotic limit $\eta(v)$ which is given by the average in $x$ of $\omega(x,v). $ Indeed at the same $L_1$ distance from $\eta$ there exists a BGK stationary solution of the HMF model.\\
We give an example of such BGK solution,
which can be constructed using that any function of the mean-field energy is an equilibrium. In this example we consider the attractive HMF model with
\[
  \mathcal{F}[f](t,x) =
  \partial_x \Biggl{(}\int_{S^1 \times \mathbb{R}}\cos(x-y)f(t,y,v)
  \, dy \, dv \Biggr{)} 
\]
in \eqref{HMF} and we choose, for $\beta,\nu>0$ to be fixed, 
\[
\omega_{\beta, \nu}(x,v)=\frac{\e^{-\beta H_\nu(x,v)}}{\mathcal{Z}},
\]
where $H_\nu(x,v)=\frac{v^2}{2}-\nu \cos x$ and $\mathcal{Z}$ is the normalizing constant.
 Using the simple structure of the potential, we have that $\omega_\nu(x,v)$ is a stationary solution of the attractive HMF model,
provided that the following compatibility condition is fulfilled:
\[
\Omega_\beta(\nu)\equiv\int \omega_{\beta,\nu}(x,v) \cos x \, dx \, dv = \nu .
\]
By Taylor expansion $\Omega_\beta(\nu)=\beta \nu/2+o(\beta \nu)$ as $\nu \to 0$, while $\Omega_\beta(\nu) \to 1$ if $\nu \to +\infty$. Hence for $\beta>2$ there exists at least one value $\widebar{\nu}$ such that $\Omega_\beta(\widebar{\nu})=\widebar{\nu}$.
\end{Remark}
\begin{Remark}
In section 2 we have proved exponential damping of solutions of the HMF model 
in the scattering setting in the perturbative case, 
while in this section we prove the result for $\tau$ large. 
These two sections could have been partially joined by considering as a smallness parameter $\epsilon=e^{-\lambda \tau}.$
However, given the different nature of the problems faced, we believe it is clearer to derive the two results separately.
\end{Remark}

\section{The Cauchy problem}
In this section, instead of fixing an asymptotic condition, we study the
Cauchy problem
for equation \eqref{HMF}, with initial condition at time zero.
We refer to Section 5 for the discussion of the differences
and advantages of the backward approach compared to this.
Putting \eqref{differenziale} in integral form we get
\begin{equation}
\label{hforward}
\begin{split}
\widehat{h}_n(t, \xi)&=\widehat{h}_n(0, \xi)+\delta_{n, \pm 1}n\frac{ \imm}{2} \int_0^t \zeta_n(s)\widetilde{\eta'}(\xi-ns)\, ds\\
&- \frac{\eps}{2} \sum_{k= \pm 1}k\int_0^t \zeta_k(s)\widehat{h}_{n-k}(s, \xi-ks)(\xi -ns) \, ds,
\end{split}
\end{equation}
and taking $\xi=nt$ for $n=\pm1$ in \eqref{hforward}, we obtain the equation
for the electric field:
\begin{equation}
\label{zforward}
\begin{split}
\zeta_n(t)&=\widehat{h}_n(0, nt)+n\frac{\imm}{2} \int_0^t \zeta_n(s)\widetilde{\eta'}(n(t-s))\, ds\\
&- \frac{\eps}{2} \sum_{k= \pm 1}kn\int_0^t \zeta_k(s)\widehat{h}_{n-k}(s, nt-ks)(t-s) \, ds.
\end{split}
\end{equation}
We introduce the weight $A^{\lambda, p}_n(\xi)=e^{\lambda \left \langle n, \xi \right \rangle} \left \langle n, \xi \right \rangle^p$
and the corresponding analytic norm of a generic function $f$ as
\[
\| f \|_{\lambda, p}=\sup_{n,\xi}A^{\lambda, p}_n(\xi)|\widehat f_n(\xi)|.
\]
In the following we take a mean-zero initial datum $h_0$ such that $\|h_0\|_{\lambda_0, p}<+\infty$, for some $\lambda_0$ and $p$ to be fixed.

As done before, we want to study the coupled system $(\zeta_{\pm 1},h)$. For this purpose,
we define the norm of the electric field $\zeta$ as
\begin{equation}
\label{avanticampo}
J_{\lambda_0}^p[\zeta]= \sup_{\beta(\lambda, t)>0} \e^{\lambda t}
\left \langle t \right \rangle^p |\zeta_{\pm1}(t)|.
\end{equation}
Here
\begin{equation}
  \label{eq:beta}
  \beta(\lambda, t)=\lambda_0 -\lambda - \delta \arctan(t)
\end{equation}
with $\delta<2\lambda_0/\pi$ measures the loss of analytic regularity
with respect to $\lambda_0$.

We remark that the choice of the $\arctan$ function 
is not mandatory, contrary to the case in Section 2,
in which the regularity decay is more precisely prescribed
by the structure of the estimates.

We define a weighted-in-time norm on $h$ with two terms:
\begin{equation}
\label{avantinorma}
K^{3,p+1}_{\lambda_0, q}[h]=\mathcal{K}^3[h]+K^{p+1}_q[h],
\end{equation}
where
$$
\mathcal{K}^3[h]=\sup_{\beta(\lambda, t)>0} \|h(t)\|_{\lambda, 3}
$$
and
$$
K^{p+1}_q[h]=\sup_{\beta(\lambda, t)>0} \beta(\lambda, t)^{1/2} \frac{\|h(t)\|_{\lambda, p+1}}{\left \langle t \right \rangle^q}.
$$
The occurrence of the last term is in the spirit of the abstract
Cauchy-Kovalevskaya theorem,
while the term $\mathcal{K}^3$  
is due to the 
treatment of the two echoes term in the equation
for $\zeta_{\pm 1}$, as we show in Prop. \eqref{propo:zita}.

\subsection{Estimates for $\zeta$}
In the sequel, for $\gamma>\lambda_0$,  it is useful to introduce the quantity
\begin{equation}
\label{jei}
j_n(t)=\imm \frac{n}{2}\widetilde{\eta'}(nt) \e^{\lambda_0 t}
\end{equation}
and define
\begin{equation}
\label{ordereps}
G_{\eps}(t) \equiv  \widehat{h}_n(0, nt) - \frac{\eps}{2} \sum_{k= \pm 1}\int_0^t \zeta_k(s)\widehat{h}_{n-k}(s, nt-ks)kn(t-s) \, ds.
\end{equation}
\begin{lemma}
Let $\eta(v)$ analytic such that $\|\eta'\|_\gamma<+\infty$ with $\gamma>\lambda_0$. If
$$
\mathcal{L}[j_1](\sigma) \neq 1 \quad \text{for} \quad \Re \sigma \ge 0
$$
then
\begin{equation*}
J_{\lambda_0}^p[\zeta] \le C(\gamma, \lambda_0) J_{\lambda_0}^p[G_{\eps}].
\end{equation*}
\end{lemma}
\begin{proof}
Assume $p=0$ and take $\lambda>0$ such that $\beta(\lambda,t)>0$ then
$$
\e^{\lambda t} \zeta_1(t)= \int_0^t \jmath_{\lambda}(t-s)\e^{\lambda s} \zeta_1(s) \, ds + \e^{\lambda t} G_{\eps}(t)
$$
with $\jmath_{\lambda}(t)\equiv\e^{-(\lambda_0-\lambda) t} j_1(t)$. 
From Theorem \eqref{resolvo},
since $\jmath_{\lambda} \in L_1(\mathbb{R}_+)$ for $\gamma>\lambda_0$ and
$$\mathcal{L}[\jmath_{\lambda}](\sigma)=\mathcal{L}[j_1](\sigma+\lambda_0-\lambda) \neq 1 \quad \text{for} \quad \Re \sigma \ge 0,$$
there exists a unique resolvent kernel $r_\lambda$ associated to $\jmath_{\lambda}$
with  $r_{\lambda} \in  L_1(\mathbb{R}_+)$.
Doing the convolution with  $r_{\lambda}$, we get
$$
\e^{\lambda t} \zeta(t)= \int_0^t r_{\lambda}(t-s) \e^{\lambda s} G_{\eps}(s) \, ds + \e^{\lambda t}G_{\eps}(t).
$$
Taking the absolute value, we obtain
\begin{equation}
\label{inversion}
\e^{\lambda t} |\zeta(t)| \le (1+  \| r_{\lambda} \|_1 ) J_{\lambda_0}^0[G_{\eps}],
\end{equation}
and we get the thesis for $p=0$ taking the supremum over $\beta(\lambda, t)>0$.

Let us give the proof in the case $p=1$, which it is not difficult to
extend to the general one.
$$
t \e^{\lambda t} \zeta(t)= \int_0^t \jmath_\lambda(t-s)
s\e^{\lambda s} \zeta(s) \, ds + Z_{\eps}(t)
$$
with
$$
Z_{\eps}(t)\equiv \int_0^t \jmath_\lambda(t-s)
(t-s)\e^{\lambda s}\zeta(s)+t\e^{\lambda t} G_{\eps}(t).
$$
Using \eqref{inversion}, we get
$$
J_{\lambda_0}^1[\zeta] \le C(\gamma,\lambda_0) \sup_{\beta(\lambda, t)>0}|Z^{\eps}(t)|
$$
and 
$$
| Z^{\eps}(t)| \le C(\gamma,\lambda_0) J_{\lambda_0}^0[\zeta]+ J_{\lambda_0}^1[G_{\eps}] \le C(\gamma, \lambda_0) J_{\lambda_0}^1[G_{\eps}],
$$
using again \eqref{inversion}.
\end{proof}

\begin{Proposition}
  \label{propo:zita}
In the hypothesis of the previous lemma, let $p\ge q+3$ with $q \ge 3$ fixed. Given $h(x,v,t)$ such that $K^{3,p+1}_{\lambda_0, q}[h] < +\infty$ we have
\begin{equation*}
J_{\lambda_0}^p[\zeta] \le C + \eps C J_{\lambda_0}^p[\zeta]K^{3, p+1}_{\lambda_0, q}[h].
\end{equation*}
\end{Proposition}
\begin{proof}
From the previous lemma, we only need to estimate $J_{\lambda_0}^p[G_\eps]$.
Multiplying by $\e^{\lambda t}\left \langle t \right \rangle^p$ in \eqref{ordereps} and using $\left \langle t \right \rangle^p \le C\Bigl{(}\left \langle t-s \right \rangle^p+\left \langle s \right \rangle^p \Bigr{)}$ we have
$$
e^{\lambda t} \left \langle t \right \rangle^p |G_{\eps}[\zeta_1](t)| \le \|h(0)\|_{\lambda_0, p}+ \eps ( I_1+ I_2)
$$
where
\begin{equation*}
\begin{split}
I_1&=\int_0^t z_{\lambda, p}(s) \e^{\lambda (t-s)} \Bigl{(} |\widehat h_0(s,t-s)|(t-s) + |\widehat h_2(s,t+s)| (t-s) \Bigr{)}\,ds\\
&\le \int_0^t z_{\lambda, p}(s)\| h(s)\|_{\lambda, 3} \Biggl{(} \frac{1}{\left \langle t-s \right \rangle^{2}} + \frac{1}{\left \langle t+s \right \rangle^{2}} \Biggr{)} \, ds
\end{split}
\end{equation*}
and 
\begin{equation*}
\small
\begin{split}
I_2&=\int_0^t z_{\lambda, p}(s) \e^{\lambda (t-s)}\frac{ \left \langle t-s \right \rangle^p}{\left \langle s \right \rangle^p}\Bigl{(} |\widehat h_0(s,t-s)|(t-s) + |\widehat h_2(s,t+s)| (t+s) \Bigr{)}\, ds\\
&\le \int_0^t z_{\lambda, p}(s) \|h(s)\|_{\lambda, p+1} \frac{1}{\left \langle s \right \rangle^{p}} \, ds.
\end{split}
\end{equation*}
Thus we obtain,
\begin{equation*}
I_1 \le J_{\lambda_0}^p[\zeta] \mathcal{K}^3[h] \int_0^t  \Biggl{(} \frac{1}{\left \langle t-s \right \rangle^{2}} + \frac{1}{\left \langle t+s \right \rangle^{2}} \Biggr{)} \, ds \le C  J_{\lambda_0}^p[\zeta] \mathcal{K}^3[h]
\end{equation*}
while, if $p-q \ge 2$,
$$
I_2\le  J_{\lambda_0}^p[\zeta] K^{p+1}_q[h]\int_0^t \frac{1}{\left \langle s \right \rangle^{p-q}\beta^{1/2}(\lambda,s)}\,ds \le C  J_{\lambda_0}^p[\zeta] K^{p+1}_q[h]
$$
and this concludes the proof.
\end{proof}

\subsection{Estimates for h}

We start by showing how to split the term with $|\xi - ns|$ in \eqref{hforward}. 
\begin{lemma}
Let $\xi \in \mathbb{R}$, $p \in \N$, $n \in \Z$ and $\lambda>0$ then
\begin{equation}
\label{alge}
A^{\lambda, p}_n(\xi)|\xi-ns| \le C\Biggl{(} A^{\lambda, p+1}_{n-k}(\xi-ks) A^{\lambda,1}_{1}(s) + A^{\lambda, 1}_{n-k}(\xi-ks) A^{\lambda, p+1}_1(s)\Biggr{)}
\end{equation}
with $k=\pm1$.
\end{lemma}
\begin{proof}
We notice that
$$
|\xi-ns|=|\xi-ks+(k-n)s| \le\left \langle s \right \rangle \left \langle n-k,\xi-ks \right \rangle.
$$
Using the triangular inequality 
$$\left \langle n,\xi \right \rangle \le \left \langle n-k,\xi-ks \right \rangle + \left \langle k,ks \right \rangle,
$$
the fact that $$
\Bigl{(} \left \langle n-k,\xi-ks \right \rangle + \left \langle k,ks \right \rangle\Bigr{)}^p \le C\bigl{(}\left \langle n-k,\xi-ks \right \rangle^p + \left \langle k,ks \right \rangle^p\Bigr{)}
$$
and $k=\pm1$, we get \eqref{alge}.
\end{proof}

We now turn to estimate equation \eqref{hforward}. As usual, we define
\[
D_n(t,\xi)=\delta_{n, \pm 1}n\frac{ \imm}{2} \int_0^t \zeta_n(s)\widetilde{\eta'}(\xi-ns)\, ds.
\]
\begin{lemma}
Given $\zeta_{\pm 1}(t)$, for $\lambda, q \ge 0$ we have
\begin{equation}
\label{delta}
\begin{split}
\| h(t) \|_{\lambda, q} &\le \|h_0 \|_{\lambda_0, q}+ \|D(t)\|_{\lambda, q}\\
&+\eps \int_0^t z_{\lambda, q+1}(s)\| h(s) \|_{\lambda, 1}+  z_{\lambda, 1}(s) \| h(s) \|_{\lambda, q+1}\, ds.
\end{split}
\end{equation}

\end{lemma}
\begin{proof}
Multiplying by $A^{\lambda,q}_n(\xi)$ in \eqref{hforward} and using \eqref{alge}, we get
\begin{equation*}
\begin{split}
  A^{\lambda,q}_n(\xi)&|h_n(t,\xi)| \le \|h_0\|_{\lambda_0, q}+A^{\lambda,q}_n(\xi)
  |D_n(t,\xi)|\\
&+ \eps \sum_{k= \pm1} \int_0^t A^{\lambda,1}_{1}(s) |\zeta_k(s)|A^{\lambda, q+1}_{n-k}(\xi-ks)|h_{n-k}(s,\xi-ks)|\, ds \\
&+  \eps \sum_{k= \pm1} \int_0^tA^{\lambda, q+1}_1(s) |\zeta_k(s)|A^{\lambda, 1}_{n-k}(\xi-ks)|h_{n-k}(s,\xi-ks)|\, ds.
\end{split}
\end{equation*}
 Since $\e^{\lambda \left \langle 1, s \right \rangle} \left \langle 1, s \right \rangle^q \le C \e^{\lambda s}  \left \langle s \right \rangle^q$, after taking the supremum over $n,\xi$ we obtain the thesis.
\end{proof}

\begin{Proposition}
  Let $p\ge q+3$ with $q \ge 3$ fixed. Given $\zeta_{\pm1}$ such that $J_{\lambda_0}^p[\zeta]
  <+\infty$ we have
\begin{equation*}
K^{3,p+1}_{\lambda_0, q}[h]\le C\|h_0\|_{\lambda_0,p}+ C  J_{\lambda_0}^p[\zeta] \|\eta'\|+ \eps C\Bigl{(}1+\frac{1}{\delta}\Bigr{)} J_{\lambda_0}^p[\zeta]K^{3,p+1}_{\lambda_0, q}[h].
\end{equation*}

\end{Proposition}
\begin{proof}
We first estimate the term of order one in \eqref{hforward}. If $m \ge p$, 
\begin{equation}
\label{stimaerre}
\begin{split}
A^{\lambda, m}_n (\xi) |D_n(t, \xi)| &\le C J_{\lambda_0}^p[\zeta] \|\eta'\| \int_0^t \e^{-(\gamma - \lambda)\left \langle \xi-ns \right \rangle} \left \langle s \right \rangle^{m-p} \left \langle \xi - ns \right \rangle^p \, ds, \\
&\le  C J_{\lambda_0}^p[\zeta] \|\eta'\| \left \langle t \right \rangle^{m-p}
\end{split}
\end{equation}
where we have used that $A^{\lambda, q}_n(\xi) \le C A^{\lambda, q}_n(\xi-ns) A^{\lambda, q}_n(ns)$ and the hypothesis on $\eta'$.\\
Now, since the norm \eqref{avantinorma} is composed by two parts, we start giving an estimate of the $\mathcal{K}^3$ norm. Using the result in \eqref{delta} we obtain
\begin{equation*}
\| h(t) \|_{\lambda, 3} \le \|h(0) \|_{\lambda_0, p} +\|D(t)\|_{\lambda, 3}+ \eps J_{\lambda_0}^p[\zeta]  \int_0^t \frac{\mathcal{K}^3[h]}{\left \langle s \right \rangle^{p-4}}+ \frac{K_q^{p+1}[h]}{\left \langle s \right \rangle^{p-1-q}}\, ds.
\end{equation*}
Using \eqref{stimaerre}, we get
$$
\mathcal{K}^3[h] \le \|h(0) \|_{\lambda_0, p}+ C J_{\lambda_0}^p[\zeta] \|\eta'\|_\gamma +\eps C J_{\lambda_0}^p[\zeta] K_{\lambda_0, q}^{3, p+1}[h].
$$
Next, we focus on $K_q^{p+1}$. Using \eqref{delta} with $p+1$, we get
\begin{equation*}
\begin{split}
\|h(t)\|_{\lambda, p+1} \le C\|h(0) \|_{\lambda_0, p}+\|D(t)\|_{\lambda, p+1}
+ \eps J_{\lambda_0}^p[\zeta](A_1+A_2).
\end{split}
\end{equation*}
where
$$
A_1= \int_0^t  \left \langle s \right \rangle^2  \| h(s) \|_{\lambda, 1}\, ds \le
C \left \langle t \right \rangle^{3} \mathcal{K}^3[h], \quad A_2=\int_0^t \frac{\| h(s) \|_{\lambda, p+2}}{\left \langle s \right \rangle^{p-1}}\, ds.
$$
For what concern $A_2$ we take
$$
\lambda'(s)=\frac{\lambda_0 - \delta \arctan(s)-\lambda}{2}
$$
then
$$
\|h(s)\|_{\lambda, p+2} \le \frac{\| h(s) \|_{\lambda', p+1}}{\lambda'-\lambda}
$$
and we get the bound
$$
A_2\le C \int_0^t \frac{1}{\left \langle s \right \rangle^{p-q-1}}\frac{K_q^{p+1}[h]}{\beta^{3/2}(\lambda,s)} \, ds
\le \frac{C}{\delta} \frac{K_q^{p+1}[h]}{\beta^{1/2}(\lambda, t)}
$$
where we have used that $p \ge q+3$ and the fact that the integral is exactly computable by 
$$
\frac{\de}{\de t} \beta^{-1/2}(\lambda, t)=
\frac{\delta}{2} \frac{1}{\beta^{3/2}(\lambda, t) \left \langle t \right \rangle^2}.
$$
Then we get, using $q \ge 3$,
\begin{equation}
\label{termuno}
\frac{\beta(\lambda, t)^{1/2}}{\left \langle t \right \rangle^q}\eps J_{\lambda_0}^p[\zeta](A_1+A_2)\le \eps J_{\lambda_0}^p[\zeta]\Bigl{(}C \mathcal{K}^3[h]+\frac{C}{\delta}K_q^{p+1}[h]\Bigr{)}.
\end{equation}
It remains to estimate the term of order one $D_n(t,\xi)$.
Using \eqref{stimaerre}, we obtain
\begin{equation}
\label{termdue}
\frac{\beta(\lambda, t)^{1/2}}{\left \langle t \right \rangle^q}\|d(t)\|_{\lambda,p+1}\le C  J_{\lambda_0}^p[\zeta] \|\eta'\|.
\end{equation}
Collecting the terms in \eqref{termuno} and \eqref{termdue} we conclude the proof.
\end{proof}

\subsection{The forward result}
\begin{Theorem}
Let us fix $p \ge q+3$ with $q \ge 3$ and consider $h_0(x,v)\in L^1(S^1 \times \R)$ a mean-zero analytic initial perturbation such that $\|h_0\|_{\lambda_0,p}<+\infty$ for some $\lambda_0>0$. Let $\eta(v)\in L^1(\R)$ analytic such that $\|\eta'\|_{\gamma}<+\infty$ with $\lambda_0<\gamma$.
Moreover, assume
\[
\mathcal{L}[j_1](\sigma) \neq 1 \quad \text{if} \quad \Re \sigma \ge 0,
\]
with $j_1$ as in \eqref{jei}.
Then there exists a unique solution $h(x,v,t)$ of \eqref{principale}
with initial datum $h_0$ such that $K^{3, p+1}_{\lambda_0, q}[h]<+\infty$ and exist $h_{\infty}$ with $\|h_{\infty}\|_{\widebar{\lambda},p}<+\infty$ for $\widebar{\lambda}<\lambda_0-\delta \pi/2$  such that
\[
\lim_{t \to \infty} \| h(x,v,t)-h_{\infty}(x,v)\|_{\infty}=0
\]
with exponential rate.
\end{Theorem}
\begin{proof}
The proof is analogous to the first part of Theorem \eqref{MAIN}. By a standard iterative procedure as in \eqref{iterative1} and using the smallness of the parameter $\eps$, we get the existence of the unique solution $h$ in the class of functions such that $K^{3, p+1}_q <+\infty$.
Then the damping property follows from the estimate
$$
\| \partial_t h(t)\|_{0,p} \le C \e^{-\widebar{\lambda} t}
$$
with $\widebar{\lambda}<\lambda_0 - \delta \pi/2$. It follows  that $h(t) \to h_\infty$ with exponential rate. 
\end{proof}

\section{Backward vs forward}

In the scattering problem,
the decay of the analytic
regularity, in the spirit
of the abstract Cauchy-Kovalevskaya theorem,
is more difficult to establish
(compare the definition of $\alpha^T(\mu,t)$
in \eqref{domain}, \eqref{eqdiff} with that of
$\beta(\lambda,t)$ in  \eqref{eq:beta}).
Despite this fact, the scattering approach is easier.
In particular,
the bound on the norm \eqref{campoindietro} 
guarantees that for any $t\ge 0$
$$|\zeta_{\pm 1}(t)|\le c \e^{-\lambda t},$$
while the bound on the norm \eqref{avanticampo}
guarantees an estimate with a time correction:
for any $t\ge 0$ and $\lambda < \lambda_0 - \delta
\arctan t$
$$|\zeta_{\pm 1}(t)|\le c \e^{-\lambda t}/\langle t \rangle ^p.$$

More in general, the norm on $h$ in 
\eqref{analytic}, \eqref{CK} is simpler than 
that in \eqref{avantinorma},
in which we 
have to introduce algebraic weights
like $\left \langle t \right \rangle^q$ in order to obtain closed estimates. 

This technical issue is mainly due to the 
different treatment of the plasma echoes,
the resonances which
occur in 
\eqref{zetaequa} and \eqref{zforward}
when $nt=ks$, {\it i.e. } when $n=k=\pm 1$,
and $t=s$.
In the {\it a-priori} estimate of $\zeta_{\pm 1}$
in Proposition \ref{stimadelcampo}, there are no difficulties
and we control the resonant terms, those with 
$k=n$, in the same way as the non-resonant ones, those with $k=-n$.
In Proposition \ref{propo:zita}, the echoes force us to introduce
the additional term $\mathcal{K}^3$ in the norm of $h$.
Note also that, in \eqref{MAIN}, we perform a more subtle control of the echoes
in \eqref{dasplittare}, with an estimate in two time steps,
by using \eqref{accazero} and the mean zero of $\omega-\widebar{\omega}$.
In this way, we obtain the backward non-perturbative result of Section 3.

\vskip.3cm
The main reason of this different behavior is that the solution $h(t)$,
with asymptotic datum $h_\infty$,
{\it gains} regularity
as $t$ increases, thanks to the damping properties of the free flow,
while the solution $h(t)$, 
with initial datum $h_0$, {\it loses} regularity as $t$ increases.
The non-perturbative result clarifies this point:
in some sense
for $t\in [\tau,+\infty)$, for large $\tau$, the evolution
is close to the free flow and it is not much affected by the echoes.
In the forward problem,
at finite time, despite the pertubative setting,
the free flow regularizing property
has not yet acted, then the effect of the echoes
is more challenging.

These plasma echoes
are considered the major technical difficulty in obtaining global in time estimates for this kind of equations.
We believe that the difference in the echo treatment is the main advantage of
the backward approach. 
 This issue is confirmed by the comparison with the other works in the literature which deal with the forward problem. In \cite{FR}[ eq.s (2.1), (2.2)] an analogous term is introduced to treat the two modes of the electric field; also in the general case in \cite{BMM}[eq.s (2.12 a/b/c)]the norm is chosen in order to control the so-called reaction and transport terms of the equation.

 \section*{Appendix: proof of lemma \eqref{lemma:a}}

Here we omit the symbol $\delta$ from $a_{T,\delta}$.
Since $a_{T }(t)$ is decreasing, we have, for any $\widebar{t} \in [0,T]$,
$$\begin{aligned}
  a_{T }(0)&=a_{T}(\widebar{t})+\delta \int_0^{\widebar{t}}e^{-a_{T }(s)s}(1+s) \, ds \le a_{T}(\widebar{t})+\delta\int_0^{\widebar{t}}e^{-a_{T}(\widebar{t})s}(1+s) \, ds\\
  &\le a_{T}(\widebar{t})+\delta\frac{1}{a_{T}(\widebar{t})}+\delta\frac{1}{a_{T}^2(\widebar{t})}.
\end{aligned}
$$
If $\delta \le 1$, the minimum of $x+\delta/x+ \delta/x^2$, for $x>0$, 
is less then $c_1\delta^{1/3}$ and is reached in $x<c_2\delta^{1/3}$.
Then,  if $a_{T }(0)\ge \max(c_1,c_2)\delta^{1/3}$,  
the right-hand side reach the minimum for some $\bar t$, and then
$a_{T }(0) \le c_1\delta^{1/3}$. This implies that 
 $a_{T }(0) \le \max(c_1,c_2) \delta^{1/3}$.

For any $t<T$, $a_{T }$ is uniformly bounded and is
increasing in $T$, so it  converges to a positive
function $a_{\infty}(t)$.
For any time interval in $[0,+\infty)$,
by dominated convergence 
in the integral
formulation of \eqref{eqdiff},
we get that $a_{\infty}(t)$
solves the differential equation
with initial datum $a_{\infty}(0)$.

Now we prove that $\lim_{t\to +\infty} a_{\infty}(t) = 0$.
First notice that given $b>0$ there exists $b_0>0$ such that
the solution of
$$\dot a = -\delta \e^{-ta}(1+t)$$
with initial datum $b_0$ exists for all times and
$a(t)\ge b$ for all time. To prove this,
we choose $b_0> b + \delta(1/b + 1/b^2)$
and consider the first time $\tau$ such that $a(\tau)= b$.
Until $\tau$,
$$b_0 - a(t) = \delta \int_0^t \e^{-as}(1+s)\de s \le
\delta \left( \frac 1b +\frac 1{b^2}\right).$$
Then $\tau = +\infty$.

Let $a(0)$ be the initial datum of a  generic solution $a(t)$. Set
$$\bar a = \inf\{ a(0)| \lim_{t\to +\infty} a(t) \ge 0\},$$
and let $\bar a(t)$ the solution with initial datum $\bar a$.
It is easy to prove that
$\bar a(t)\to 0$, otherwise $\bar a$ is
not the infimum.
We conclude the proof by noticing that
$a_\infty(0) \le \bar a$, then
$a_\infty(t)$ is dominated by $\bar a(t)$ which is a vanishing function.

\end{document}